\documentclass[11pt]{amsart} %
\usepackage{amsmath,amssymb, amstext, enumerate}
\usepackage{fullpage}
\usepackage{graphicx}
\usepackage{graphics}
\usepackage{psfrag}

\newtheorem{theorem}{Theorem}[section]

\newtheorem{proposition}[theorem]{Proposition}
\newtheorem{corollary}[theorem]{Corollary}
\newtheorem{lemma}[theorem]{Lemma}
\newtheorem{example}[theorem]{Example}

\newtheorem{question}[theorem]{Question}

\numberwithin{equation}{section}

\def\IC{{\bf C}}

\def\b0{{\bf 0}}

\def\u{{\bf u}}
\def\x{{\bf x}}
\def\y{{\bf y}}

\def\cN{{\mathcal N}}
\def\cV{{\mathcal V}}

\def\conv{{\textrm{conv}}\,}

\def\tr{{\rm tr}\,}

\def\({\left(}
\def\){\right)}
\def\[{\left[}
\def\]{\right]}

\def\tr{{\rm tr}}
\def\qed{\hfill\vbox{\hrule width 6 pt
\hbox{\vrule height 6 pt width 6 pt}}}

\newcommand{\glim}{\operatorname{glim}\,}

\begin{document}
%\openup 1\jot
%\maketitle{}

\title[Spectrum and numerical ranges]{The spectrum of the product of operators,\\
and the product of their numerical ranges}

\author{%David Li-Wei Kuo,
Chi-Kwong Li,  Ming-Cheng Tsai,\\ Kuo-Zhong Wang and Ngai-Ching Wong
}

\address[Li]{Department of Mathematics, College of William \& Mary, Williamsburg, VA 23187, USA.}
\email[Li]{ckli@math.wm.edu}

\address[%Kuo,
Tsai and Wong]{Department of Applied Mathematics, National Sun Yat-sen University,
Kaohsiung 80424, Taiwan.}
%\email[Kuo]{mpu.verilog@gmail.com}
\email[Tsai]{mctsai2@gmail.com}
\email[Wong]{wong@math.nsysu.edu.tw}

\address[Wang]{Department of Applied Mathematics, National Chiao Tung University, Hsinchu 30010, Taiwan.}
\email[Wang]{kzwang@math.nctu.edu.tw}

\date{\today}

\dedicatory{This paper is dedicated to Professor  Pei Yuan Wu}

\maketitle

\begin{abstract}
We show that a compact operator $A$ is a multiple of a positive semi-definite operator
if and only if
$$
\sigma(AB) \subseteq \overline{W(A)W(B)}, \quad\text{for all (rank one) operators $B$}.
$$
An example of a normal operator is given to show
that the equivalence conditions may fail in general.
We then obtain conditions to identify other classes of operators $A$ so that
equivalence conditions hold.
\end{abstract}

\maketitle

\section{Introduction}

Let $B(H)$ be the algebra of bounded linear operators on a complex Hilbert space $H$.
We identify $B(H)$ with
$M_n$, the algebra of $n\times n$ complex matrices, if $H$ has finite dimension $n$.
The {\em spectrum}  $\sigma(A)$,
%the {\em spectral radius} $r(A)$,
and the {\em numerical range} $W(A)$
%and the {\em numerical radius} $w(A)$
of an operator $A\in B(H)$ are defined by
$$
\sigma(A) = \{\lambda:  A-\lambda I   \text{ is not invertible}\}, \quad \text{ and } \quad
%r(A) &= \max\{|z|:z\in\sigma(A)\},\\
W(A) = \{\langle Ax, x\rangle : x\in H, \|x\|=1\},$$
%w(A)& = \sup \{|z| : z\in W(A)\}, \end{align*}
respectively.
Here, $\langle\cdot,\cdot\rangle$ and $\|\cdot\|$ denote the inner product and its
corresponding norm of $H$.

The Hausdorff--Toeplitz theorem asserts that $W(A)$  is always a
bounded convex subset of the complex plane.  When $H$ is finite dimensional, it is compact.
In general, the closure of the numerical range satisfies
$$
\sigma(A) \subseteq \overline{W(A)}.
$$
When $A$ is normal, we have
$$
\conv\sigma(A) = \overline{W(A)}.
$$
Here, $\conv S$ denotes the convex hull of a set $S$ in a vector space.
The spectrum and the numerical range are useful tools for studying operators and
matrices.
Motivated by the theoretical development and applications,
researchers have obtained many interesting results;
see, for example,
\cite{G97}, \cite[Chapter 22]{H82} or \cite[Chapter 1]{H91}.

In perturbation theory, one might want to estimate $\sigma(A+B)$ for
``small'' $B$, but it is known that
$\sigma(A+B) \not \subseteq \sigma(A) + \sigma(B)$.
For example, let $A = \begin{pmatrix} 0 & M \cr 0 & 0 \cr\end{pmatrix}$ and
$B = \begin{pmatrix} 0 & 0 \cr \varepsilon & 0 \cr\end{pmatrix}$
with positive $M, \varepsilon > 0$.
Then $\sigma(A) = \sigma(B)= \{0\}$, whereas
$\sigma(A+B) = \{ \pm \sqrt{M\varepsilon}\}$.
Nevertheless, we always have
$$\sigma(A+B) \subseteq W(A+B) \subseteq W(A) + W(B).$$
Thus, $W(A)+W(B)$ provides a containment region for $\sigma(A+B)$.

In (multiplicative) perturbation theory, one considers $\tilde{A}=AB$ such that $B$
is closed to the identity operator $I$.  However, %for (composition) product of operators,
neither of the inclusion
$\sigma(AB) \subseteq \sigma(A)\sigma(B)$ nor $W(AB) \subseteq W(A)W(B)$ holds.
The following example in \cite{HO88} tells us that the above inclusions may not hold even for $2\times 2$ hermitian matrices $A, B$.
Let
$$
A=\left(
         \begin{array}{cc}
           1 & 0 \\
           0 & -1 \\
         \end{array}
       \right)
\quad\text{and}\quad
B= \left(
     \begin{array}{cc}
       0 & 1 \\
       1 & 0 \\
     \end{array}
   \right).
$$
Then
$$
\sigma(AB) = \{\pm\sqrt{-1}\}\not\subseteq W(A)W(B) = [-1,1].
$$
Nevertheless, it was shown in \cite{W67}
that if $A,B\in B(H)$ and $0\not\in\overline{W(A)}$, then
$$
\sigma(A^{-1}B)\subseteq {\overline{W(B)}} / {\overline{W(A)}}.
$$
It follows from this result that if $A \in M_n$
is a multiple of a positive semi-definite matrix,
$$
\sigma(AB)\subseteq W(A)W(B) \quad\hbox{ for all } B \in M_n.
$$
In \cite[Theorem 3]{B71}, it was shown that if $A\in B(H)$ is a (multiple of a) positive
semi-definite operator, then
$$
\conv\sigma(AB)\subseteq {\overline{W(A)}}{\overline{W(B)}} \quad\hbox{ for all } B\in  B(H).
$$
%Here,  $\conv S$ denotes the convex hull of the set $S$ in the complex plane $\mathbb C$.

It is natural to consider the converse problem; namely,
\begin{question} \label{question}
Is $A\in B(H)$ a multiple of a positive semi-definite matrix if
$$
\sigma(AB)\subseteq \overline{W(A)}\overline{W(B)} \quad \hbox{\rm for all } B\in B(H)\,?
$$
\end{question}
\noindent
In \cite{A12}, the author considered this question for matrices $A \in M_n$,
and an affirmative answer to this question was claimed in the paper.
In \cite{C13}, the authors there pointed out a gap in the proof in
\cite[Theorem 2.4]{A12}, and repaired it.
In \cite{Cet}, the authors there studied inequalities
related to the spectral radius and the numerical radius of products of matrices.
In the last section of \cite{Cet}, they
pointed out another problem in \cite{A12},
and fixed it using the results in their paper.

In this paper, we consider Question \ref{question} for
infinite dimensional operators.
In particular, we refine the finite dimensional result
to the following.

\begin{theorem}\label{thm2}
Suppose $A \in B(H)$ is a compact operator.
The following conditions are equivalent.
\begin{enumerate}
\item[\text{{\bf (A1)}}]
 $A$ is a multiple of a positive semi-definite operator.

\item[\text{{\bf (A2)}}]
$\sigma(AB) \subseteq \overline{W(A)W(B)} \quad\text{ for all } B \in B(H)$.

\item[\text{{\bf (A3)}}]
$\sigma(AB) \subseteq \overline{W(A)W(B)} \quad\text{ for all rank one } B \in B(H)$.
\end{enumerate}
\end{theorem}

Of course, it would be nice to further extend the result to general operators.
However, the following example shows that it is impossible even on a separable Hilbert
space.
(A verification of the example will be given in the next section).

\begin{example}\label{eg}
Consider a Hilbert space $H$ with a countable orthonormal basis
$\{f_1, f_2, \dots \}$. Suppose $\{\mu_1, \mu_2, \dots\}$
is a dense set of the unit circle $\{e^{it}: t \in [0, 2\pi)\}$
and  $T$  is the diagonal operator on $H$ satisfying
$Tf_n = \mu_n f_n$. Then  $A := I+T$, which is not a multiple of positive semi-definite
operator, satisfies
$$
\sigma(AB) \subseteq \overline{W(A)W(B)}, \quad \hbox{ for all } B \in B(H).
$$
\end{example}

A key step in the proof of the finite dimensional result is to show that
if $A \in M_n$ satisfies condition $({\bf A2})$ in Theorem \ref{thm2}, then $A$ is normal.
We can modify Example \ref{eg} to show that this implication is not true
for general operators in the following.
(The verification will also be done in the next section.)

\begin{example} \label{eg2}
Suppose $\hat A = A \oplus \begin{pmatrix} 1 & d \cr 0 & 1 \cr \end{pmatrix}
\in B(H \oplus \IC^2)$,
where $A$ is defined as in Example \ref{eg} and $d \in (0, 1]$.
Then $\hat A$ is not normal and condition {\bf (A2)} in Theorem \ref{thm2} holds.
\end{example}

Apart from Examples \ref{eg}  and \ref{eg2}, we obtain the following
theorem, which allows us to identify other classes of operators $A$
such that the conditions {\bf (A1), (A2), (A3)} are equivalent.

\begin{theorem}\label{thm}
Let $H$ be a Hilbert space of finite or infinite dimension.
Consider the following conditions for an operator $A \in B(H)$.
\begin{itemize}
\item[\text{{\bf (A1)}}] $A$ is a multiple of a positive (semi-definite) operator.

\item[\text{{\bf (A2)}}]  $\sigma(AB) \subseteq \overline{W(A)W(B)}$ for all $B \in B(H)$.

\item[\text{{\bf (A3)}}]  $\sigma(AB) \subseteq \overline{W(A)W(B)}$ for all rank one
$B \in B(H)$.
\end{itemize}
\noindent
Then the following implications hold:
$$\text{{\bf (A1)}} \Rightarrow \text{{\bf (A2)}} \Rightarrow \text{{\bf (A3)}}.$$
The implication
$$
\text{{\bf (A3)}} \Rightarrow \text{{\bf (A1)}}
$$
holds when there is a boundary point $\mu$  of  $\overline{W(A)}$ attaining the numerical radius $|\mu|=w(A)$
and lying on two different support lines of $\overline{W(A)}$.
\end{theorem}

By Theorem \ref{thm}, we have the following.
\iffalse
In \cite{A12}, it is shown that every finite square matrix $A$ satisfying \textbf{(A3)} is normal.
But we do not know whether  it holds for infinite dimensional operators.
Assuming that $A$ is normal to start with, however, the following results will look better.
\fi

\begin{corollary}\label{cor}
In each of the following cases, conditions {\bf (A1), (A2), (A3)} in Theorem \ref{thm} are equivalent
for an operator $A \in B(H)$.
\begin{enumerate}
\item[{\rm (1)}]
${\overline{W(A)}}$ is a convex polygon, which may degenerate  to a line segment or
a point. This covers the cases when $A$ is a scalar multiple of a hermitian operator,
or when $A$ is a normal operator with finite spectrum.
\item[{\rm (2)}] $A\in B(H)$ is normal and   there is an isolated point
$\lambda$ in $\sigma(A)$ attaining the spectral radius $|\lambda|=r(A)$.
\end{enumerate}
\end{corollary}

\section{Proofs and auxiliary results}

We focus on the proof of Theorem \ref{thm}, and deduce Theorem \ref{thm2} and Corollary
\ref{cor} as consequences.  We then verify
Examples \ref{eg} and \ref{eg2}.

The implication {\bf (A1)} $\Rightarrow$ {\bf (A2)} in Theorem \ref{thm}
is a result in \cite{B71}. Furthermore, it was shown that if {\bf (A1)} holds,
then $\overline{W(A)W(B)}$ is always convex.  We give a short proof of the result.

\begin{proposition}\label{new-1}
Suppose $A \in B(H)$ is a multiple of a positive semi-definite operator.
Then for any $B \in B(H)$, the set $\overline{W(A)W(B)}$ is convex, and
$$\sigma(AB) \subseteq  \overline{W(A)W(B)}.$$
\end{proposition}

\it Proof. \rm
Without loss of generality, we can assume
that $A$ is positive semi-definite.
Note that  $\overline{W(A)} = [ a_1, a_2]$, $a_2 \ge a_1 \ge 0$,
and $\overline{W(B)}$ is a compact convex set. Thus,
$$\overline{W(A)W(B)} = \bigcup_{\mu \in \overline{W(A)}} \mu \overline{W(B)}
= \bigcup_{a_1 \le t \le a_2} t\overline {W(B)}$$
is convex.

Now, suppose $\lambda\in \sigma(AB)$. If $\lambda = 0$, then
$AB$ is singular, so that $A$ is singular or $B$ is singular.
Hence, $0 \in \sigma(A) \in \overline{W(A)}$ or
$0 \in \sigma(B) \in \overline{W(B)}$ so that
$0 \in \overline{W(A)W(B)}$.

If $\lambda \ne 0$, then
$\lambda \in \sigma(AB)\setminus\{0\} = \sigma(A^{1/2}BA^{1/2})\setminus\{0\}$.

Assume first that $\lambda$ is an   approximate eigenvalue of $A^{1/2}BA^{1/2}$.
Then there is a sequence of unit vectors $\{x_n\}$ such that
$$\|A^{1/2}BA^{1/2} x_n - \lambda x_n\| \rightarrow 0.$$
We may assume that $A^{1/2} x_n \ne 0$ for all $n \in \mathbb N$.
Thus $t_n\langle By_n, y_n\rangle \rightarrow \lambda$ in $\overline{W(A)W(B)}$
with

\medskip
\centerline{$y_n = A^{1/2}x_n/\|A^{1/2}x_n\|$ and $t_n = \langle Ax_n, x_n\rangle \in W(A)$.}

Assume next that $\lambda$ is not in the approximate point spectrum, and thus
$A^{1/2}BA^{1/2} - \lambda$ does not have a dense range.
Consequently, we can find a norm one element $y$ in $H$ orthogonal to its range.
In particular,
$$
\langle (A^{1/2}BA^{1/2} - \lambda)y, y\rangle = 0.
$$
This gives
$$
\lambda = \langle A^{1/2}BA^{1/2}y, y\rangle = t\langle Bx,x\rangle,
$$
and $A^{1/2}y\neq 0$.
Here, $t=\|A^{1/2}y\|^2 = \langle Ay, y\rangle \in W(A)$, and
$x = A^{1/2}y/\|A^{1/2}y\|$ is of norm one.  Hence, $\lambda\in W(A)W(B)$.
\qed

The implication ({\bf A2)} $\Rightarrow$ ({\bf A3}) is clear.
We now focus on the condition under which the implication
({\bf A3}) $\Rightarrow$ ({\bf A1}) holds.

\begin{proposition}
\label{prop2} Suppose $A \in B(H)$ satisfies {\bf (A3)}. Then
there is $\mu \in \overline{W(A)}$ such that $|\mu| = w(A)$.
Moreover, if such a  $\mu$ lies on two different support lines of $\overline{W(A)}$,
then {\bf (A1)} holds.
\end{proposition}

We need some preliminaries to prove Proposition \ref{prop2}.
Let $A \in B(H)$ satisfying \textbf{(A3)}. Note that conditions \textbf{(A1), (A2)} and \textbf{(A3)} will not be affected by replacing $A$ with
$\gamma UA^\ddag U^*$ for
any nonzero $\gamma$, unitary $U\in B(H)$, $A^\ddag \in \{A, A^t, A^*\}$.
We will use this fact in our proof.
%We have the following lemma.

\begin{lemma} \label{lemma} Suppose $A \in B(H)$ satisfies {\bf (A3)}.
\begin{itemize}
\item[{\rm (1)}] The operator $A$ is radialoid. That is, $r(A)=w(A)=\|A\|$.
\item[{\rm (2)}] Suppose $\eta \in W(A)$ satisfies $|\eta| = w(A)$.  Then
$A$ is unitarily similar to the orthogonal sum $\eta I \oplus A_1$ where $w(A_1) \le w(A)$
and $W(A_1) \subseteq \{\lambda: |\eta-\lambda| \le w(A)\}$. Consequently,
$|\eta -\mu| \le w(A)$ for all  $\mu \in W(A)$.
\end{itemize}
\end{lemma}

\begin{proof} We may replace $A$ by $A/\|A\|$ and assume that $\|A\| = 1$.
To prove (1),
suppose $Ax_n = \lambda_n y_n$ for some unit vectors
$x_n, y_n \in H$ with  positive scalars $\lambda_n\uparrow 1$.
Let $B_n\in B(H)$ be the rank one operator
$z\mapsto \langle z,y_n\rangle x_n$. Then $AB_ny_n = \lambda_ny_n$, and thus
$\lambda_n \in \sigma(AB_n) \subseteq \overline{W(A)W(B_n)}$. Since $w(B_n)\leq 1$,
we have $\lambda_n\leq w(A)$. That is, $1\leq w(A)$. Thus $w(A)=1$, and hence
$r(A)=1$ (see \cite[Theorem 1.3-2]{G97}).

Next, consider (2).
Suppose $\eta \in W(A)$ with $|\eta| = 1$. Then for any unit vector
$x$ such that $\langle Ax, x\rangle = \eta$, we write
$Ax = \eta x + \nu y$ for some unit vector $y$ orthogonal to $x$.
Then
$$
|\eta|^2 = \|A\|^2 \ge \|Ax\|^2 = |\eta|^2 + |\nu|^2.
$$
Thus, $\nu = 0$ and $Ax=\eta x$. Similarly, we can show that $A^*x = \bar \eta x$. The first assertion follows.

For the second assertion, we may replace $A$ by $A/\eta$ and
assume that $A=\begin{pmatrix}
1 & 0 \cr
0 & A_1 \cr\end{pmatrix}$.
Let $\lambda\in W(A_1)$.
We can assume that the leading (upper left) $2\times 2$
submatrix of $A$ is
$\begin{pmatrix}
1 & 0 \cr
0 & \lambda \cr\end{pmatrix}$.
Let  $U =\frac{1}{\sqrt{2}}\begin{pmatrix}
1 & 1 \cr
1 & -1 \cr\end{pmatrix}$ be  the $2\times 2$ unitary matrix.  Then
the leading $2\times 2$ submatrix of
$\hat A = (U\oplus I)A(U\oplus I)^*$  equals
$A_0 = \frac{1}{2}\begin{pmatrix}
1+\lambda & 1-\lambda\cr
1-\lambda & 1+\lambda\cr\end{pmatrix}$.
Let $B = (U\oplus I)^*\left(\begin{pmatrix}
0 & 0 \cr
2 & 0 \cr\end{pmatrix} \oplus O \right)(U\oplus I)$.
Then $\tr (AB) = 1-\lambda$ is the nonzero eigenvalue
of the rank one matrix $AB$.
Since $\sigma(AB) \subseteq \overline{W(A)W(B)} \subseteq \{z\in \IC: |z| \le 1\}$,
we have $|1-\lambda| \leq 1$.

The last assertion follows from the fact that
$W(A) = \conv(\{1\} \cup W(A_1))$.
\end{proof}

In \cite{A12}, the authors showed that if a matrix $A$ satisfies {\bf (A2)}, then
there is  $\mu\in {W(A)}$ satisfying $|\mu|=\|A\|$, and tried to
prove that {\bf (A1)} holds.  Lemma \ref{lemma}(1) shows that for any
$A \in B(H)$ satisfying ({\bf A3}),
there is  $\mu\in \overline{W(A)}$ such that $|\mu|=\|A\|$ and $\mu$ is an eigenvalue of $A$.

We will use Lemma \ref{lemma} to prove Proposition \ref{prop2}.
In the finite dimensional case,
$W(A)=\overline{W(A)}$ is compact, and there are unit vectors attaining
the norm of $A$.  However, it might not be the case if the underlying Hilbert space $H$
is
infinite dimensional.  Nevertheless,  we can use the Berberian construction
(see \cite{B62}) to overcome this technicality.

In connection to our problem, we will impose additional requirement in the Berberian construction, namely, we will need a generalized Banach limit which is multiplicative.  We include some details of the construction for completeness. We identify the space $\ell_\infty$ of bounded scalar sequences with  the $C^*$-algebra
%$C^b({\mathbb N})$ of bounded continuous functions on  the natural numbers $\mathbb N$, or
$C(\beta\mathbb{N})$ of
continuous functions on the Stone-Cech compactification
$\beta \mathbb N$ of $\mathbb N$.  Here,
a bounded sequence $\lambda=(\lambda_n)$ in $\ell_\infty$ corresponds
to a function $\hat{\lambda}$ in  $C(\beta\mathbb{N})$
with $\hat{\lambda}(n)=\lambda_n$ for all $n=1,2,\ldots$.
Take any point $\xi$ from $\beta\mathbb{N}\setminus \mathbb{N}$.
The point evaluation $\lambda\mapsto \hat{\lambda}(\xi)$ of $\ell_\infty$ gives a
nonzero multiplicative
generalized Banach limit, denoted by $\glim$, that  satisfies the following conditions.
For any bounded sequences $(a_n)$ and $(b_n)$ in $\ell_\infty$ and scalar $\gamma$, we have
\begin{enumerate}[(a)]
    \item $\glim (a_n+b_n) = \glim (a_n) + \glim (b_n)$.
    \item $\glim (\gamma a_n) = \gamma \glim (a_n)$.
    \item $\glim (a_n) = \lim a_n$ whenever $\lim a_n$ exists.
    \item $\glim (a_n) \ge 0$ whenever $a_n \ge 0$ for all $n$.
    \item $\glim (a_nb_n) = \glim (a_n)\glim (b_n)$.
\end{enumerate}
Equivalently, we can define $\glim (a_n) = \lim_{\mathfrak U} a_n$ through a free ultrafilter $\mathfrak U$ on $\mathbb N$,
when we consider $\beta \mathbb{N}$ consisting of ultrafilters on $\mathbb N$ and those outside $\mathbb N$ are free
(i.e. $\bigcap \mathfrak{U} = \emptyset$). Note that
all multiplicative generalized Banach limits on $\ell_\infty$ arise from the above construction.
Note also that we do not assume the translation invariant property on $\glim$.
Indeed, the only translation invariant multiplicative generalized Banach limit is zero.

Denote by $\cV$ the set of all bounded
sequences $ \{x_n\}$ with $x_n \in  H$. Then $\cV$ is a vector
space relative to the definitions $\{x_n\}+\{y_n\} = \{x_n + y_n\}$
and $\gamma \{x_n\} = \{\gamma x_n\}$.  Let $\cN$
be the set of all   sequences $\{x_n\}$ such that $\glim (\langle x_n,x_n\rangle) = 0$.
Then $\cN$ is a
linear subspace of $\cV$. Denote by $\x $ the coset $\{x_n\}+\cN$. The
quotient vector space $\cV/\cN$ becomes an inner product space with
the inner product
$\langle \x ,\y \rangle  = \glim (\langle x_n,y_n\rangle)$.
Let $K$ be the completion of $\cV/\cN$. If
$x \in H$, then $ \{x\}$ denotes the  constant sequence defined by
$x$. Since $\langle\x,\y\rangle =  \langle x,y\rangle $ for $\x  =
\{x\}+\cN$ and $\y  = \{y\}+\cN$, the mapping $x \mapsto \x $ is an isometric
linear map of $H$ onto a closed subspace  of $K $
and $K$ is an
extension of $H$. For an operator $A \in  {\mathcal B}(H)$, define
$$A_0(\{x_n\}+\cN) = \{Ax_n\} + \cN.$$
We can extend $A_0$ on $K$, which will be denoted by $A_0$ also.  The
mapping $\phi: {\mathcal B}(H)\to  {\mathcal B}(K)$ given by $\phi(T)= \tilde T$  is
a unital   isometric
$*$-representation with $\sigma(T)=\sigma(\tilde T)$.
Moreover, the approximate eigenvalues of $T$ (and also $\tilde T$)
will become eigenvalues of $\tilde T$. See \cite{B62}.

It is clear that rank one operators in $B(H)$ become rank one operators in $B(K)$.
However, rank one operators in $B(K)$ does not necessarily come from rank one operators in $B(H)$.
A counter example can be given by the rank one operator $e\otimes e$ defined by
$k\mapsto \langle k, e\rangle e$
for a nonzero vector $e$ in $K$ orthogonal to $H$.
Nevertheless, in connection to our study, we have the following.

\begin{lemma}
Let $\tilde A \in B(K)$ be the extension of $A \in B(H)$ in the Berberian construction.
Suppose $\sigma(AB) \subseteq \overline{W(A)W(B)}$ for all rank one $B \in B(H)$.
Then
$$
W(\tilde A) = \overline{W(A)}, \quad\text{and}\quad
\sigma(\tilde A B') \subseteq W(\tilde A)W(B') \text{ for all rank one } B' \in B(K).
$$
\end{lemma}

\begin{proof}
Clearly, we have  $W(\tilde A) \subseteq \overline{W(A)}$.
On the other hand, if $\langle Ax_n,x_n\rangle \rightarrow \mu \in \overline{W(A)}$,
then $\mu=\glim (\langle Ax_n,x_n\rangle) =
\langle \tilde A\x,\x\rangle\in W(\tilde A)$ with $\x$ arising from the sequence $\{x_n\}$
of unit vectors in $H$.
Thus, $W(\tilde A) = \overline{W(A)}$.

To prove the second assertion, we
make some simple observations.

\begin{enumerate}[{1.}]

\item  $W(B) = \overline{W(B)}$ for any rank one $B \in B(H)$.

\iffalse
\item
If $a_nb_n \in \overline{W(A)W(B)}$ and $\{a_nb_n\}$ converges to $\mu$,
we may consider subsequences and assume that
$\{b_n\} \rightarrow b \in \overline{W(B)} = W(B)$
and $\{a_n\} \rightarrow a \in \overline{W(A)}$.
\fi

\item
If $\x, \y, \u \in K$ correspond to the sequences $\{x_n\}$, $\{y_n\}$ and $\{u_n\}$ of unit vectors
 in $H$, then
$$
\langle (\x\otimes \y)\u,\u\rangle  = \langle \u,\y\rangle\langle \x,\u\rangle
= \glim (\langle u_n,y_n\rangle) \glim(\langle x_n,u_n\rangle)
= \glim [(\langle u_n,y_n\rangle)(\langle x_n,u_n\rangle)].
$$
\end{enumerate}

%Now suppose $\tilde A \in B(K)$ is an extension of $A \in B(H)$.
Let $B' = \x \otimes \y$ with  $\x, \y$ in $K$
arising from the sequences $\{x_n\}, \{y_n\}$ of
unit vectors in $H$.
We will show that
$$
\sigma(\tilde A B') \subseteq \{ \langle\tilde A \x,\y\rangle, 0\}
\subseteq W(\tilde A)W(B').
$$
Obviously, $0 \in W(\tilde A)W(B')$.  Observe that
$$
\langle\tilde A \x, \y\rangle
= \glim  (\langle Ax_n,y_n\rangle),
$$
where
$$
\langle Ax_n,y_n\rangle\in \sigma(A(x_n\otimes y_n)) \subseteq \overline{W(A) W(x_n\otimes y_n)}
=  W(\tilde A) W(x_n \otimes y_n)
$$
by the fact that $\overline {W(A)} = W(\tilde A)$ and 1. above.  It follows that
$$
\langle Ax_n,y_n\rangle = \langle\tilde A v_n,v_n\rangle \langle (x_n\otimes y_n)u_n,u_n\rangle
= \langle\tilde A v_n,v_n\rangle \langle u_n,y_n\rangle \langle x_n,u_n\rangle
$$
for some unit vector $v_n \in K$ and unit vector $u_n \in H$.
%By boundedness, we can assume that all three inner products in the last term above converge.
By 2. above,
$$
\glim [(\langle u_n,y_n\rangle)(\langle x_n,u_n\rangle)]=
%\glim(u_n,y_n)\glim(x_n, u_n) =
\langle(\x\otimes \y)\u,\u\rangle\in
W(\x\otimes \y)=W(B').
$$
By the compactness of  $W(\tilde A)$ and $W(B')$, we have
$$
\glim \langle Ax_n,y_n\rangle = \glim (\langle\tilde A v_n, v_n\rangle) \glim [(\langle u_n,y_n\rangle)(\langle x_n,u_n\rangle)]
\in  W(\tilde A)W(B').
$$
\vskip -.3in
\end{proof}

\begin{proof}[Proof of Proposition \ref{prop2}.]
We may replace $A$ by $A/\|A\|$ and assume that $\|A\| = 1$. Furthermore,
we may apply the Berberian construction and assume that
conditions (a) -- (e) hold. For simplicity, we assume $H = K$ and $A = \tilde A$.

By Lemma \ref{lemma}, we have
$r(A) = w(A) = \|A\|=1$ and
we may assume that $A = \mu I \oplus A_1$ for some contraction $A_1$
such that $\mu$ is not an eigenvalue of $A_1$.  Without loss of generality,
we may assume that $\mu = 1$, $A = I \oplus A_1$.
We need to show that $A_1$ is positive semi-definite.
Assume that it is not the case so that
$W(A_1) \not\subseteq [0,1]$.

By our assumption, there is a support line  of $W(A)$ passing through $1$
and $1+r_1 e^{i\alpha_1}$ for some $\pi/2 < \alpha_1 < 3\pi/2$ and $r_1 \in [0,1]$.  Replacing $A$ with $A^*$ if necessary, we can also assume that $\pi/2 < \alpha_1 < \pi$ and
$$
W(A) \subseteq \{z = 1 + r e^{i\alpha}: r \in [0,1],
\ \alpha_1 \le \alpha \le 3\pi/2, \ |z| \le 1\}.
$$

%Assume first that $\pi/2 < \alpha_1 < 2\pi/2$.
Let $B = B_0 \oplus O$ with
$$B_0 = 2e^{i(\pi/2-\alpha_1)}
\begin{pmatrix} \cos \theta & 0 \cr \sin \theta & 0 \cr \end{pmatrix},$$
where $\theta \in (0, \pi/2)$ such that
$2r_1\sin\alpha_1 \ge \tan \theta.$
Observe that
\begin{align*}
(\sin\alpha_1 \cos \theta + r_1\sin\theta)^2
&= (\sin\alpha_1\cos\theta)^2 + 2r_1\sin\alpha_1 \cos \theta\sin \theta + r_1^2\sin^2\theta\\
&> (\sin\alpha_1\cos\theta)^2 + \sin^2\theta,
\end{align*}
i.e.,
\begin{equation}
\label{theta}
\sin\alpha_1 \cos \theta + r_1\sin\theta >
\sqrt{(\sin\alpha_1\cos\theta)^2 + \sin^2\theta}.
\end{equation}
With a suitable unitary  transform, we may assume that
$A$ has a leading $2\times 2$  submatrix
$A_0 = \left(
         \begin{array}{cc}
           1 & 0\\
           0 & 1+r_1e^{i\alpha_1} \\
         \end{array}
       \right)$.
Let $U = U_0 \oplus I$
with $U_0 = \frac{1}{\sqrt{2}}\begin{pmatrix} i & -i \cr 1 & 1\cr
\end{pmatrix}$.
Then
\begin{eqnarray*}
\lambda = \tr(UAU^*B) &=&
\tr(U_0A_0U_0^*B_0) \\
&=& 2e^{i(\pi/2-\alpha_1)}\cos\theta + r_1e^{i\alpha_1} \tr(U_0E_{22}U_0^*B_0) \\
&=& 2e^{i(\pi/2-\alpha_1)}\cos\theta + r_1e^{i\alpha_1} \tr(E_{22}U_0^*B_0U_0) \\
&=&  2\sin \alpha_1\cos \theta + r_1 \sin \theta
+ i(2\cos \alpha_1\cos\theta+r_1\cos\theta),\end{eqnarray*}
which is the nonzero eigenvalue of the rank one matrix $UAU^*B$.

To derive a contradiction, we will show that
\begin{equation}\label{lambda}
\lambda \notin  W(A)W(B) = \bigcup_{z\in W(A)} W(zB).
\end{equation}
Recall that for any compact operator, and thus any finite matrix, $T$, the right support line of
$W(T)$ is the set of complex numbers with
real part equal to the maximum eigenvalue of $(T+T^*)/2$.
For each $z = 1+re^{i\alpha} \in W(A)$ with $r \in [0,1]$ and
$\alpha_1 \le \alpha < 3\pi/2$,
the maximum eigenvalue of the matrix
$$\frac{1}{2}
\left(zB_0 + \overline{z} B_0^* \right)=
\begin{pmatrix} (\sin\alpha_1 + r \sin(\alpha_1-\alpha))2\cos\theta
& -i\cdot e^{i\alpha_1}\cdot\overline{z} \sin \theta \cr i\cdot e^{-i\alpha_1}\cdot z \sin \theta
 & 0\cr
\end{pmatrix}$$
equals
\begin{equation} \label{gamma}
\gamma + \sqrt{\gamma^2  + |z|^2 \sin^2\theta} \qquad \hbox{ with }
\ \gamma = (\sin \alpha_1  + r\sin(\alpha_1-\alpha))\cos \theta.
\end{equation}
Because $\pi/2 < \alpha_1 \le \alpha  < 3\pi/2$,
we have $\alpha_1 - \alpha \in (-\pi,0)$ and
$\sin(\alpha_1-\alpha) \le 0$.
Suppose $\gamma \ge 0$, i.e., $\sin\alpha_1 \ge |r\sin(\alpha_1-\alpha)|$.
Then by (\ref{theta})
we have
$$\gamma + \sqrt{\gamma^2 + |z|^2 \sin^2\theta}
\le  \sin\alpha_1 \cos\theta +
\sqrt{(\sin\alpha_1\cos\theta)^2 +  \sin^2\theta}
< 2\sin\alpha_1 \cos\theta + r_1\sin\theta.$$
If $\gamma < 0$, i.e., $\sin\alpha_1 < |r\sin(\alpha_1-\alpha)|$,
then by (\ref{theta}) we have
$$\gamma + \sqrt{\gamma^2 + |z|^2\sin^2\theta} < |z| \sin\theta <
2\sin\alpha_1 \cos\theta + r_1\sin\theta.$$
Thus, the real part of every point in  $W(zB)$
is strictly less than $2\sin\alpha_1\cos \theta + r_1 \sin\theta$,
and not equal to $\lambda$. Since this is true for any $z \in W(A)$,
we get the desired contradiction.
\end{proof}

\medskip

\begin{proof}[Proof of Theorem \ref{thm2}]
We want to show that the implication {\bf (A3)} $\Rightarrow$ {\bf (A1)}
is valid when $A$ is  compact.
Let $A \in B(H)$ be compact satisfying {\bf (A3)}.
Since $A$ is compact, every nonzero element in $\sigma(A)$ is an eigenvalue of $A$. Hence we have $\sigma(A)\setminus\{0\}\subseteq  W(A)$.
In view of Lemma \ref{lemma}, we can assume that $\|A\|=r(A)=w(A)=1$ which  is an eigenvalue of $A$, and write
$A = I \oplus A_1$ such that $\|A_1\| \leq 1$ and
$1 \notin \sigma(A_1)$.
Note that   the
largest eigenvalue $\lambda$ of the compact  operator $(A_1+A_1^*)/2$ is less than $1$.
Indeed, if there is a unit vector $x$ such that $(A_1+A_1^*)x/2=x$, then the inequality
$$1=\left\langle \frac{A_1+A_1^*}{2}x, x\right\rangle=\frac{1}{2}\langle A_1x, x\rangle+ \frac{1}{2} \langle A_1^*x, x\rangle \leq1$$
implies that $1=\langle A_1x, x\rangle$ and hence $A_1x=x$,  contradicting that $1 \notin \sigma(A_1)$.
Since
$$
W(A_1) \subseteq \{ \nu \in \IC: |\nu|\leq 1, |1-\nu| \le 1, \text{ and } (\nu + \bar \nu)/2 \leq \lambda\},
$$
and $W(A) = \conv(\{1\} \cup W(A_1))$, we see that there are two
different support lines of $W(A)$ passing through $1$.
It then holds  {\bf (A1)} by Proposition \ref{prop2}.
\end{proof}

\begin{proof}[Proof of Corollary \ref{cor}]
The assertions in (1) are clear.
For (2), let
$A\in B(H)$ be normal satisfying {\bf (A3)} and let, without loss of generality, $1=\|A\|$ be an
isolated point in the spectrum $\sigma(A)$ of $A$.
Write
$A = I_1 \oplus  A_1$
(an orthogonal sum), where  $I_1$ is the eigen-projection  of  $A$ for $1$, and $1\notin \sigma(A_1)$.
Moreover, we can separate $1$ from $\sigma(A_1)$ by a straight line in the complex plane.
Consequently, $1$ and $\overline{W(A_1)}=\conv{\sigma(A_1)}$ are contained in two disjoint open half spaces.
Since $W(A)$ is the convex hull of the set
$\{1\}\cup W(A_1)$,
we see that $1$ lies on two different support lines of $\overline{W(A)}$.
Proposition \ref{prop2} applies and finishes the proof.
\end{proof}

\begin{proof}[Verification of Example \ref{eg}.]
By the Berberian construction, we may assume that
$A = I+T$ such that $T$ is normal and every point $e^{it}$ on the unit
circle is an eigenvalue.
Suppose  that $\lambda \in \sigma(AB)$. The case $\lambda = 0$ is done, since
$0 \in W(A)W(B)$ as $0 \in W(A)$.
Suppose $\lambda \ne 0$. Because
$\sigma(AB)$ and $\sigma(BA)$ have the same nonzero elements, we see that
$\lambda \in \sigma(BA)$.

Assume first that $\lambda$ is an approximate eigenvalue of $BA$.
By the Berberian construction, we may assume that there is a unit vector $x$ such that
$BAx = \lambda x$.  Let $Ax = a_{11} x + a_{21} y$
such that $a_{11} = \langle Ax,x\rangle$ and $y$
is a unit vector  orthogonal to $x$.
Because $A-I$ is unitary, we have
\begin{align}\label{eq:a11}
|a_{11}-1|^2 + |a_{21}|^2 = 1.
\end{align}
Using an orthonormal basis with $x, y$ as the first two vectors,
and abusing notations for matrices of uncountable sizes,
we see that the operator matrices of $A$ and $B$ have the form
$$\begin{pmatrix} a_{11} & * & * \cr a_{21} & * & * \cr 0 & * & * \cr\end{pmatrix}
\quad \hbox{ and } \quad
\begin{pmatrix} B_1 & * \cr * & *  \cr\end{pmatrix} \quad
\hbox{ with } \quad  B_1 =
\begin{pmatrix} b_{11} & b_{12} \cr b_{21} & b_{22} \cr\end{pmatrix}.
$$
Then $W(B_1) \subseteq W(B)$,
$$\lambda = b_{11}a_{11} + b_{12} a_{21} \quad \hbox{ and } \quad
b_{21} a_{11} + b_{22} a_{21} = 0.$$
It follows from \eqref{eq:a11} that $a_{11}-1$ lies in the closed unit complex disk.  Hence
we can write
$$
a_{11} -1 = -\alpha + (1-\alpha)e^{ir}
$$
for some $r \in [0, 2\pi)\setminus\{\pi\}$ and some $\alpha\in [0,1]$.

Let $A_0 = \left(
         \begin{array}{cc}
           0 & 0\\
           0 & 1+e^{ir} \\
         \end{array}
       \right)$
be the compression of $A$ on the
two dimensional subspace spanned by $\{f_{\pi}, f_r\}$.
Let $u=\sqrt{\alpha} f_\pi+ \sqrt{1-\alpha} f_r$.  Then
$$
\langle (A_0 - I_2)u,u\rangle = a_{11}-1.
$$
Because $A_0-I_2$ is unitary, $\|(A_0 - I_2)u\|=1$.  In view of \eqref{eq:a11},
we see that $A_0-I_2$ is unitarily similar to a matrix of the form
$\begin{pmatrix} a_{11}-1 & * \cr a_{21} & * \cr\end{pmatrix}.$
Hence, $A_0$ is unitarily similar to
$A_1 = \begin{pmatrix} a_{11} & * \cr a_{21} & * \cr\end{pmatrix}$
and $W(A_1) = W(A_0) \subseteq W(A)$.
Note that $B_1A_1$ is in upper triangular form with $\lambda$ lying in the
$(1,1)$ position. Thus, $\lambda \in \sigma(B_1A_1)$.
Note that $A_0$, as well as $A_1$, is a multiple of positive semi-definite matrix.
By the implication \textbf{(A1)} $\Rightarrow$ \textbf{(A2)} in Theorem \ref{thm},
we have
$$
\lambda \in \sigma(B_1  A_1) \subseteq W(B_1)W(A_1)\subseteq \overline{W(B)W(A)}.
$$

At this point, we have shown that $\overline{W(B)W(A)}$
contains all approximate eigenvalues of $BA$.
Let $\alpha\in \sigma(BA)$ and $\alpha$ is not an approximate eigenvalue of $BA$.
Then $\alpha$ is in the
interior of $\sigma(BA)$ and thus there is an approximate eigenvalue $\lambda$ of $BA$,
which is a boundary point of $\sigma(BA)$,
such that $\alpha = \beta\lambda$ with $0<\beta<1$.  Since
$\lambda\in \overline{W(B)W(A)}= \bigcup_{|z-1|\leq 1} z\overline{W(B)}$,
we have $\alpha = \beta zb$ for some $b\in \overline{W(B)}$ and $z$ satisfying $|z-1|\leq 1$.
Since $|\beta z - 1| \leq \beta |z-1| + (1-\beta) \leq 1$,
we have $\alpha\in \overline{W(B)W(A)}$ as well.
\end{proof}

\begin{proof}[Verification of Example \ref{eg2}.]
For any  $B$ in $B(H)$, we
show that $\sigma(\hat AB) \subseteq \sigma(B\hat A)\cup\{ 0 \}  \subseteq W(\hat A)W(B)$. Since $0 \in W(\hat A)$, we have $0 \in W(\hat A)W(B)$. So, we focus on those nonzero $\lambda \in  \sigma(B\hat A)$.

Similar to the Verification of Example \ref{eg},
we only need to consider the case when $\lambda$ is a
nonzero approximate eigenvalue of $B\hat{A}$.
Using a similar argument as in the  Verification of Example \ref{eg}, we may
assume that the operator matrices of $\hat A$ and $B$ have the form
$$\begin{pmatrix} a_{11} & * & * \cr a_{21} & * & * \cr 0 & * & * \cr\end{pmatrix}
\quad \hbox{ and } \quad
\begin{pmatrix} B_1 & * \cr * & *  \cr\end{pmatrix} \quad
\hbox{ with } \quad  B_1 =
\begin{pmatrix} b_{11} & b_{12} \cr b_{21} & b_{22} \cr\end{pmatrix}.
$$
Then $W(B_1) \subseteq W(B)$, and
$$\lambda = b_{11}a_{11} + b_{12} a_{21} \quad \hbox{ and } \quad
b_{21} a_{11} + b_{22} a_{21} = 0.$$
Because $\hat A - I$ is a contraction, we see that
$$|a_{11}-1|^2 + |a_{21}|^2 \le 1.$$
We can then construct a unitary matrix $A_0 \in M_3$
with first column equal to $(a_{11}-1, a_{21}, a_{31})^t$, where
$a_{31}=(1-|a_{11}-1|^2-|a_{21}|^2)^{1/2}$.
Since $A-I$ is a unitary operator with spectrum $\{e^{it}: t \in [0, 2\pi)\}$,
we may regard $A_0$ as a compression of $A-I$, and hence
$I+A_0$ is a compression of $A$ and can be viewed as
the leading principal submatrix of $UAU^*$, whose first column has only three
nonzero entries, namely, $a_{11}, a_{21}, a_{31}$. So, the first column of
$(B_1\oplus [0])(I+A_0)$ equals $(\lambda, 0,0)^t$, and thus the first column of
$(B_1 \oplus O)(UAU^*)$ has only one nonzero entry $\lambda$ lying in the $(1,1)$ position.
Since $A$ satisfies \textbf{(A2)}, we have
\begin{align*}
\lambda &\in \sigma((B_1 \oplus O)(UAU^*)) = \sigma((U^*(B_1 \oplus O)U)A)\\
&\subseteq \sigma(A(U^*(B_1 \oplus O)U))\cup\{0\}\subseteq W(A)W(U^*(B_1 \oplus O)U)
= W(A)W(B_1 \oplus O).
\end{align*}
Note that
$W(B_1 \oplus O) = \conv\{W(B_1) \cup W(O)\} = \{r b: b \in W(B_1), r \in [0,1]\}$.
By the convexity of $W(A)$ and the fact that $0 \in W(A)$, if
$a \in W(A)$ and $r \in [0,1]$, then $ra \in W(A)$.
Therefore, $\lambda \in W(A)W(B_1\oplus O)$ implies that
$\lambda = a(rb)$ with $a \in W(A)$, $b \in W(B_1) \subseteq W(B)$, $r \in [0,1]$.
It follows that  $\lambda = (ra)b \in W(A)W(B) = W(\hat A)W(B)$.
\end{proof}

\medskip\noindent
%{\bf Acknowledgment}

\section*{acknowledgment}
Li is an honorary professor of the University of Hong Kong and Shanghai University.
His research was  supported by US NSF and HK RCG. This project was done while he was visiting the National Sun Yat-sen University, the Hong Kong Polytechnic University in January and February of 2014. He would like to thank the colleagues of these universities for their warm hospitality.

The Research was supported by the Ministry of Science and Technology of the Republic
of China under the projects NSC 102-2115-M-009-006 (for  Wang) and 102-2115-M-110-002-MY2
(for  Tsai and Wong).

We thank Che-Man Cheng for sending us the preprint \cite{Cet}.

\end{document}